\documentclass[a4paper,11pt]{article}
\usepackage{amsmath}
\usepackage{amssymb}
\usepackage{amscd}
\usepackage{amsthm}
\usepackage{flafter}
\usepackage{float}
\usepackage{tikz}
\usetikzlibrary{matrix,decorations.pathreplacing,calc}

\pgfkeys{tikz/mymatrixenv/.style={decoration=brace,every left delimiter/.style={xshift=3pt},every right delimiter/.style={xshift=-3pt}}}
\pgfkeys{tikz/mymatrix/.style={matrix of math nodes,left delimiter=[,right delimiter={]},inner sep=2pt,column sep=1em,row sep=0.5em,nodes={inner sep=0pt}}}
\pgfkeys{tikz/mymatrixbrace/.style={decorate,thick}}
\newcommand\mymatrixbraceoffseth{0.5em}
\newcommand\mymatrixbraceoffsetv{0.2em}

\newcommand*\mymatrixbraceright[4][m]{
	\draw[mymatrixbrace] ($(#1.north west)!(#1-#3-1.south west)!(#1.south west)-(\mymatrixbraceoffseth,0)$)
	-- node[left=2pt] {#4} 
	($(#1.north west)!(#1-#2-1.north west)!(#1.south west)-(\mymatrixbraceoffseth,0)$);
}

\newcommand*\mymatrixbracebottom[4][m]{
	\draw[mymatrixbrace] ($(#1.south west)!(#1-1-#3.south east)!(#1.south east)-(0,\mymatrixbraceoffsetv)$)
	-- node[below=2pt] {#4} 
	($(#1.south west)!(#1-1-#2.south west)!(#1.south east)-(0,\mymatrixbraceoffsetv)$);
}

\newtheorem{lemma}{Lemma}

\newtheorem{theorem}[lemma]{Theorem}
\newtheorem{corollary}[lemma]{Corollary}

\newtheorem{example}[lemma]{Example}

\newcommand{\rmv}[1]{}

\newcommand{\by}{\times}

\newcommand{\ZZ} {\mathbb{Z}}
\newcommand{\FF} {\mathbb{F}}

\newcommand{\bfa}{{\bf a}}
\newcommand{\bfb}{{\bf b}}
\newcommand{\bfr}{{\bf r}}
\newcommand{\bfs}{{\bf s}}

 \author{Daniel Panario, Mark Saaltink, Brett Stevens, Daniel Wevrick \\
 School of Mathematics and Statistics, Carleton University (Panario, Stevens)  and \\
 Independent researchers (Saaltink, Wevrick)\\
 \texttt{daniel@math.carleton.ca, mark.saaltink@gmail.com,} \\
 \texttt{brett@math.carleton.ca, dwevrick@gmail.com}}
 \date{}

\begin{document}

\title{A general construction of ordered orthogonal arrays using LFSRs}

\author{Daniel Panario,Mark Saaltink, Brett Stevens and Daniel Wevrick
\thanks{Daniel Panario is with 
           School of Mathematics and Statistics,
           Carleton University, Canada.
           Email: daniel@math.carleton.ca}
\thanks{Mark Saaltink and Daniel Wevrick are unaffiliated researchers.
           Emails: mark.saaltink@gmail.com, dwevrick@gmail.com}
\thanks{Brett Stevens is with
           School of Mathematics and Statistics,
           Carleton University, Canada.
           Email: brett@math.carleton.ca}
\thanks{Daniel Panario and Brett Stevens are partially supported by NSERC of Canada.}
}

\markboth{IEEE Transactions on Information Theory}{Panario, Saaltink, 
Stevens and Wevrick.}

\maketitle

\begin{abstract}
In \cite{Castoldi}, $q^t \by (q+1)t$ ordered orthogonal arrays (OOAs) 
of strength $t$ over the alphabet $\FF_q$ were constructed using 
linear feedback shift register sequences (LFSRs) defined by 
{\em primitive} polynomials in $\FF_q[x]$. In this paper we extend 
this result to all polynomials in $\FF_q[x]$ which satisfy some 
fairly simple restrictions, restrictions that are automatically 
satisfied by primitive polynomials. While these restrictions 
sometimes reduce the number of columns produced from $(q+1)t$ to 
a smaller multiple of $t$, in many cases we still obtain the 
maximum number of columns in the constructed OOA when using 
non-primitive polynomials. For small values of $q$ and $t$, we 
generate OOAs in this manner for all permissible polynomials of 
degree $t$ in $\FF_q[x]$ and compare the results to the ones 
produced in \cite{Castoldi}, \cite{Rosenbloom} and \cite{Skriganov} 
showing how close the arrays are to being ``full'' orthogonal arrays.
Unusually for finite fields, our arrays based on non-primitive
irreducible and even reducible polynomials are closer to
orthogonal arrays than those built from primitive polynomials.
\end{abstract}

\section{Introduction} \label{Sec:Intro}
Let $t$, $m$, $v$, $\lambda \in \ZZ^+$ with $2 \leq t \leq m$, let 
$n=\lambda v^{t}$ and let $M$ be an $n \by m$ array over a set $V$ 
(the {\em alphabet}) of cardinality $v$. An $n\times t$ subarray of 
$M$ is {\em $\lambda$-covered} if each $t$-tuple over $V$ appears 
as a row of the subarray exactly $\lambda$ times; the corresponding 
set of $t$ columns of $M$ is {\em $\lambda$-covered}. An {\em 
orthogonal array} (OA) of {\em strength} $t$ and {\em index} $\lambda$ 
is an $n \by m$ array such that every set of $t$ columns is 
$\lambda$-covered. We use $OA_{\lambda}(n;t,m,v)$ to denote such 
an array.

A generalization of the concept of an OA is an {\em ordered orthogonal 
array} (OOA) \cite{Lawrence}, \cite{Mullen}. For $s \in \ZZ^+$, let 
$M$ be an $n \times ms$ array with columns labeled by $\{1,2,\ldots, m\} 
\times \{1,2,\ldots,s\}$ and entries from a set $V$ (the {\em alphabet}) 
of cardinality $v$. For $1 \le i \le m$, the set of columns indexed 
by $\{i\} \times \{1,2,\ldots,s\}$ is the {\em block} $Bl_M(i)$.
Thus,
\[
M=\left[\begin{array}{c|c|c|c}
 Bl_M(1) & Bl_M(2) & \cdots & Bl_M(m)\\
\end{array}
\right].
\]

A subset $\Omega$ of columns of $M$ is {\em left-justified} if 
$(i, j) \in \Omega$ with $j>1$ implies $(i, j-1) \in \Omega$.
Thus, for each block $Bl_M(i)$, the columns in $\Omega$ that are 
in $Bl_M(i)$ are consecutive leftmost ones in $Bl_M(i)$, with 
indices $(i,1),(i,2),\ldots, (i,l_i)$ where $l_i$ is the number 
of columns in $\Omega$ chosen from $Bl_M(i)$. We observe that 
$0 \le l_i \le s$ and so $l_i$ could be $0$, which occurs when 
no columns of $\Omega$ are from $Bl_M(i)$. If $M$ has the 
property that every left-justified set $\Omega$ of $t$ columns 
of $M$ is $\lambda$-covered, then $M$ is an {\em ordered orthogonal 
array}, and is denoted $OOA_{\lambda}(n;t,m,s,v)$.

As before, $t$ is the {\em strength} of the OOA and $\lambda$ is 
the {\em index}. For both OAs and OOAs, if $\lambda = 1$, we simply 
say {\em covered} and write $OA(n;t,m,v)$ and $OOA(n;t,m,s,v)$ to 
denote such arrays. By setting $s=1$, an $OOA_{\lambda}(n;t,m,1,v)$ 
can easily be seen to be an $OA_{\lambda}(n;t,m,v)$ as every subset 
of $t$ columns is left-justified. As a concrete example, an 
$OOA(8;3,2,3,2)$ with $2$ blocks of size $3$ over the alphabet 
$\FF_2$ is shown in Figure~\ref{Fig:fig1}.

\begin{figure}[H]
\normalsize
\centering
\begin{equation*}
OOA(8;3,2,3,2)=\left[\begin{array}{ccc|ccc}
$(1,1)$ & $(1,2)$ &  $(1,3)$ & $(2,1)$ &  $(2,2)$ & $(2,3)$  \\ \hline
0 & 0 & 1 & 1 & 1 & 0  \\
0 & 1 & 1 & 0 & 1 & 1  \\
1 & 1 & 0 & 0 & 0 & 1  \\
1 & 0 & 0 & 1 & 0 & 0  \\
0 & 1 & 0 & 1 & 0 & 1  \\
1 & 0 & 1 & 0 & 1 & 0  \\
0 & 0 & 0 & 0 & 0 & 0  \\
1 & 1 & 1 & 1 & 1 & 1  \\
\end{array}\right]
\end{equation*}
\caption{A strength $3$ OOA over alphabet $\FF_2$.}
\label{Fig:fig1}
\end{figure}

In an OA, all subsets of $t$ columns are $\lambda$-covered and in 
an OOA all left-justified subsets of $t$ columns are $\lambda$-covered.
In an OOA, it is possible that additional subsets of $t$ columns, 
those that are not left-justified, are also $\lambda$-covered.
An OOA that has a significant percentage of non left-justified
subsets of $t$ columns $\lambda$-covered can be thought of as being 
closer to an OA and as having {\em better coverage}.

We observe that the ordering and letters used for the parameters
in the notation, $OA_{\lambda}(n;t,m,v)$ and $OOA_{\lambda}(n;t,m,s,v)$,
is identical with the notation in~\cite{Castoldi}.  The order of the parameters is the same as that used in the Handbook of Combinatorial Designs \cite{Colbourn} but
may be different from other reference books and/or papers \cite{HFF}.

More information about OAs and OOAs can be found in 
\cite[Section VI.59.3]{Colbourn}, \cite{Hedayat} and 
\cite[Chapter 3]{Tamar}.

Ordered orthogonal arrays are the combinatorial analog of $(t,m,s)$-nets which were introduced by Niederreiter for their utility in quasi-Monte Carlo numerical integration~\cite{NiederMC}. Let $[0, 1)^s$ be the half-open unit cube of dimension $s$. An {\em elementary interval in base b} in $[0, 1)^s$ is a set of the form
\[
E = \prod_{i=1}^{s}\left[ \dfrac{a_{i}}{b^{d_{i}}},\dfrac{a_{i}+1}{b^{d_{i}}} \right)
\]
where, for each $i$, $d_{i} \geq 0$ and  $0\leq a_{i} < b^{d_{i}}$. The volume of $E$ is $b^{-\sum d_{i}}$.
Let $s \geq 1$, $b \geq 2$, and $m \geq t \geq 0$ be integers. A {\em $(t, m, s)$-net in base $b$} is a multiset
$\mathcal{N}$ of $b^{m}$ points in $[0, 1)^s$ with the property that every elementary interval in base $b$
of volume $b^{t-m}$ contains precisely $b^t$ points from $\mathcal{N}$. This property is used to prove that the point set of a $(t, m, s)$-net has a bounded {\em star-discrepancy} which, using the Koksma-Hlawka Inequality, in turn provides a bound on the error in the quasi-Monte Carlo numerical integration \cite{Colbourn}.  In Section~\ref{Sec:ff2} we will compute the star-discrepancy of some of the $(t, m, s)$-nets from our construction.

The relationship between $(t,m,s)$-nets and OOAs was proved by Lawrence and independently by Mullen and Schmid.
\begin{theorem}(\hspace{-0.2mm}\cite{Lawrence,Mullen})\label{tms-net}
Let $s\geq 1$, $b\geq 2$, $t\geq 0$ and $m$ be integers, and assume
that $m \geq t+1$ to avoid degeneracy. Then there exists a
$(t,m,s)$-net in base $b$ if and only if there exists an 
$OOA_{b^{t}}(b^{m};m-t,s,m-t,b)$.
\end{theorem}
If the rows of an OOA form a linear subspace, as is the case for all the OOAs constructed in this article, then the corresponding net is {\em digital}~\cite{Colbourn}.

Earlier papers giving constructions of OOAs include \cite{Castoldi}, 
\cite{Rosenbloom} and \cite{Skriganov}. Our paper generalizes the 
construction given in \cite{Castoldi} which uses linear feedback 
shift registers defined by primitive polynomials in $\FF_q[x]$. In 
our construction, primitivity is not required. Section \ref{Sec:LFSR} 
reviews, generalizes and simplifies a number of results from~\cite{Castoldi}.
A key technical ingredient in~\cite{Castoldi} is the use of trace 
representation since, in that paper, only primitive polynomials are considered. 
Here we cannot use this ingredient since no trace representation 
is readily available for non-primitive polynomials. Section~\ref{Sec:GR} 
gives the new construction and Section~\ref{Sec:proofOOA} provides 
a proof of the correctness of the construction. In place of the 
trace representation, we use the fact that the annihilating polynomial 
commutes with the left shift operator and scalar multiplication. 
In Section~\ref{Sec:ff2}, we investigate the case of polynomials 
in $\FF_2[x]$, provide statistical data on the OOAs produced and 
compare these new OOAs to the ones produced in \cite{Castoldi}, 
\cite{Rosenbloom} and \cite{Skriganov}. Similar results for larger 
field sizes are provided in Appendix~\ref{App:stat_analysis}. An 
initial analysis of the data shows that, on many occasions, the 
new construction produces OOAs that have better coverage than 
previously generated ones. In contrast to other situations in 
finite fields, here irreducible and reducible polynomials provide 
better coverage than primitive polynomials; see tables in 
Section~\ref{Sec:ff2} and Appendix~\ref{App:stat_analysis}.

\section{Linear Feedback Shift Registers} \label{Sec:LFSR}
Let $\FF_q$ be a finite field of order $q$ and let $f \in \FF_q[x]$ be 
a monic polynomial of degree $t$. Writing $f(x) = \sum_{i=0}^{t}b_ix^i$, define a linear feedback shift register (LFSR) whose output stream, 
$S(f,T) = \bfs  = (\bfs_n)_{n \ge 0}$ is given by

$$
\bfs_{n+t} = -\sum_{i=0}^{t-1}b_i\bfs_{n+i}, \qquad n \ge 0,
$$
for a given initial $t$-tuple $T = (\bfs_0,\bfs_1, \ldots, 
\bfs_{t-1}) \in \FF_q^t$.

This is equivalent to having $\sum_{i=0}^{t}b_i\bfs_{n+i} = 0$ for 
all integers $n \ge 0$. Then, $f$ {\em annihilates} such streams and
such streams are {\em annihilated} by $f$. Let $G(f)$ denote 
the set of streams over $\FF_q$ that satisfy the recurrence defined 
by $f$.

If $b_0 \ne 0$ we can ``run the LFSR in reverse'' by defining
$$
\bfs_{n-t} = -b_{0}^{-1}\sum_{i=1}^{t}b_i\bfs_{n-t+i}, \qquad n \in \ZZ,
$$
and so we can allow the indexing set for ${\bf s} $ to be $\ZZ$ and 
not just the non-negative integers. Thus, $\sum_{i=0}^{t}b_i\bfs_{n+i} = 0$ 
for all integers $n$. We also note that $\bfs$ is periodic; see 
\cite{Golomb} for more information about linear feedback shift registers.

\subsection{Operations on streams}\label{Sec:StrOp}
Fix a monic polynomial $f \in \FF_q[x]$ with non-zero constant term.
For $\bfs = (\bfs_n)_{n \ge 0} \in G(f)$ and $\beta \in \FF_q$, we 
define $L \bfs  $ to be the stream ${\bf b} = (\bfs_{n+1})_{n \ge 0}$ 
and $\beta \bfs $ to be the stream {\bf b} = $(\beta \bfs_n)_{n \ge 0}$.
Here $L$ is the {\em left-shift operator} and these operations can be 
extended in the obvious way to define $g(L){\bf s} $ for $g \in \FF_q[x]$.
From \cite{Golomb}, a stream $\bfr$ has {\it a run of zeroes of length 
$l$ at index $n$} if
$$
\bfr_{n-1} \ne 0, \bfr_n = \bfr_{n+1} = \cdots = \bfr_{n+l-1} =0,  \bfr_{n+l} \ne 0.
$$
For Lemmas $1$-$5$ we always have the following context:
\begin{itemize}
\item $f \in \FF_q[x]$ is a monic polynomial of degree $t$ with $f(x) = \sum_{i=0}^{t}b_ix^i$ and $b_0 \ne 0$;
\item $\beta \in \FF_q$ with $f(\beta) \ne 0$;
\item $\bfr, \bfs \in G(f)$ with $\bfr = (L-\beta)\bfs$;
\item $\bfr$ has a run of $l \ge 0$ zeroes at index $0$ (hence $\bfr_{l} \ne 0$).
\end{itemize}

The following two lemmas are parts of Proposition $6$ in \cite{Castoldi} and are included here for completeness.

\begin{lemma} \cite[Proposition $6$]{Castoldi} \label{Lem:betapow}
$(\bfs_{0}, \bfs_{1},  \bfs_{2}, \ldots, \bfs_{l}) = 
(\bfs_{0}, \beta \bfs_{0}, \beta^2 \bfs_{0}, \ldots, \beta^l\bfs_{0})$.
\end{lemma}

\begin{proof}
This follows from $(L-\beta)\bfs = \bfr$ and the assumption that 
$\bfr$ has a run of $l \ge 0$ zeroes at index $0$.
\end{proof}

\begin{lemma}\cite[Proposition $6$]{Castoldi} \label{Lem:zeroes}
$\bfs_{0} = 0$ if and only if $\bfs$ has a run of $l+1$ zeroes 
starting at index $0$.
\end{lemma}
\begin{proof}
This follows from Lemma~\ref{Lem:betapow}.
\end{proof}

From $\bfr = (L-\beta)\bfs$, it follows that if $\bfs$ has a run of $l+1$ zeroes at index $0$ (i.e. when $\bfs_0 = 0$), then $\bfr$ has a run of $l$ zeroes at index $0$.
Momentarily disregarding the fourth assumption that $\bfr$ has a run of $l$ zeroes at index $0$, we note that even if $\bfs_0 \ne 0$, Lemma~\ref{Lem:betapow} shows that it is still possible to obtain a run of $l$ zeroes in $\bfr$.
In \cite{Castoldi}, a special polynomial was constructed from $\bfr$ whose roots determine the exact conditions under which the run of $l$ zeroes in $\bfr$ came from a run of $l+1$ zeroes in $\bfs$.
This polynomial is:

$$
P_{l,\bfr}(x) = \sum_{j=l+1}^{t}b_j \Big(\sum_{m=l}^{j-1}\bfr_{m}x^{j-1-m}\Big).
$$

The next three lemmas prove properties of $P_{l,\bfr}$.
These results were shown as Theorem $1$ in \cite{Castoldi}, but the 
proofs given here are significantly shorter and less complicated than the ones  given there.
These lemmas are used in Section~\ref{Sec:GR}
to prove the correctness of the construction given in this paper.

\begin{lemma}\cite[Theorem $1$]{Castoldi}\label{Lem:magic}
The polynomial  $\bfs_{0}f + P_{l,\bfr}$ has $\beta$ as a
root.
\end{lemma}
\begin{proof}
For $j \in \ZZ$, since $\bfr = (L-\beta)\bfs$, we have $\bfr_j = \bfs_{j+1} - 
\beta \bfs_j$ or equivalently $\bfs_{j+1} = \bfr_j + \beta \bfs_j$.
Hence $\bfs_1 = \bfr_0 + \beta \bfs_0, \bfs_2 = \bfr_1 + \beta \bfs_1 
= \bfr_1 + \beta \bfr_0 + \beta^2 \bfs_0$, and so we have by induction
$$
\bfs_j = \bfs_0 \beta^j + \sum_{m=0}^{j-1} \bfr_m \beta^{j-1-m} 
 \quad \mbox{for all}\ j \in \mathbb{N}_0.
$$
Since $\bfs$ is annihilated by $f$, we have

\hspace{2.0in}
$0 = \sum_{j=0}^t b_j \bfs_j$

\vspace{0.1in}

\hspace{2.14in}
$= \sum_{j=0}^t b_j \Big(\bfs_0 \beta^j + \sum_{m=0}^{j-1} 
 \bfr_m \beta^{j-1-m}\Big)$

\vspace{0.1in}

\hspace{2.14in}
$= \bfs_0 \sum_{j=0}^t b_j \beta^j + \sum_{j=0}^t b_j 
\Big(\sum_{m=0}^{j-1} \bfr_m\beta^{j-1-m}\Big)$

\vspace{0.1in}
	
\hspace{2.14in}
$= \bfs_0 \ f(\beta)
+ \sum_{j=0}^t b_j\Big( \sum_{m=0}^{j-1} \bfr_m\beta^{j-1-m}\Big)$.

\noindent 
Since $\bfr_0 = \bfr_1 = \cdots = \bfr_{l-1} = 0$, the first $l$ 
terms of the inner sum are $0$ and so we have
$$
0 =  \bfs_0 f(\beta) + \sum_{j=0}^t b_j\Bigg( \sum_{m=l}^{j-1} 
\bfr_m\beta^{j-1-m}\Bigg).
$$
When $0 \le j \le l$, the inner sum is empty and so 
$$
0 =  \bfs_0 f(\beta) + \sum_{j=l+1}^t b_j\Bigg( \sum_{m=l}^{j-1} 
\bfr_m\beta^{j-1-m} \Bigg) = \bfs_0 f(\beta) + P_{l,\bfr}(\beta).
$$
Hence $\beta$ is a root of $\bfs_0 f + P_{l,\bfr}$ as claimed.
\end{proof}

\begin{lemma}\cite[Theorem $1$]{Castoldi}\label{Lem:betaroot}
The stream $\bfs$ has a run of zeroes of length $l+1$ starting at index 
$0$ if and only if $\beta$ is root of $P_{l,\bfr}(x)$ if and only if 
$\bfs_0 = 0$.
\end{lemma}
\begin{proof}
This follows from Lemmas~\ref{Lem:zeroes} and \ref{Lem:magic} and 
the fact that $f(\beta) \ne 0$.
\end{proof}
To make the next results easier to understand, we re-write 
$P_{l,\bfr}(x)$ via a number of steps.
First, reverse the inner sum to get
$$
P_{l,\bfr}(x) =  
\sum_{j=l+1}^t b_j\Bigg( \sum_{m=0}^{j-l-1} \bfr_{j-1-m}x^{m}
\Bigg).
$$

Then, switch the order of the two sums to get
$$
P_{l,\bfr}(x) =  
\sum_{m=0}^{t-l-1} \Bigg(\sum_{j=l+m+1}^t b_j \bfr_{j-1-m}\Bigg)\ x^{m}.
$$
Lastly re-index the inner sum to get
\begin{equation}\label{(1)}
P_{l,\bfr}(x) =  
\sum_{m=0}^{t-l-1} \Bigg(\sum_{j=0}^{t-m-l-1} 
b_{j+m+l+1} \bfr_{j+l}\Bigg)\ x^{m}.
% \hspace{0.5in} (1)
\end{equation}

We observe that $P_{l,\bfr}$ has degree smaller than or equal to $t-l-1$ 
and the coefficient of $x^{t-l-1}$ is $b_t\bfr_{l}$. In fact, 
$P_{l,\bfr}$ has degree $t-l-1$ since $b_t  = 1$ and $\bfr_{l} \ne 0$.

Now, if $\bfs$ has a run of $l+1$ zeros starting at index $0$, then 
we can express the polynomial $P_{l+1,\bfs}(x)$, for this tuple of 
zeroes, as

\begin{equation}\label{(2)}
P_{l,\bfr}(x) =  
\sum_{j=0}^t b_j\Bigg( \sum_{m=0}^{j-l-1} \bfr_{j-1-m}x^{m}
\Bigg).
\end{equation}

By Lemma~\ref{Lem:betaroot}, $\beta$ is a root of $P_{l,\bfr}$.
Hence $P_{l,\bfr}(x) = (x-\beta)A(x)$ for some $A \in \FF_q[x]$.
We show that $A$ is, in fact, $P_{l+1,\bfs}(x)$

\begin{lemma}\cite[Theorem $1$]{Castoldi}\label{Lem:ltolp}
Suppose $\bfs$ has a run of at least $l+1$ zeros starting at index $0$.
Then we have $P_{l,\bfr}(x) = (x-\beta)P_{l+1,\bfs}(x)$.
\end{lemma}
\begin{proof}
Let $H(x) = (x-\beta)P_{l+1,\bfs}(x) - P_{l,\bfr}(x)$ and, for $m \in \mathbb{N}_0$, consider 
$[x^m]H(x)$, the coefficient of $x^m$ in $H(x)$.
We show that it is $0$ and hence $H(x) = 0$.
The constant term of $H(x)$ is
$$
-\beta \ [x^0] P_{l+1,\bfs}(x) - [x^0] P_{l,\bfr}(x).
$$

Using $P_{l,\bfr}(x) = \sum_{j=0}^t b_j\big( \sum_{m=0}^{j-l-1} 
\bfr_{j-1-m}x^{m}\big)$ and $P_{l+1,\bfs}(x) = \sum_{j=0}^t b_j
\big( \sum_{m=0}^{j-l-2} \bfs_{j-1-m}x^{m}\big)$ which are 
equivalent formulae for $P_{l,\bfr}$ and $P_{l+1,\bfs}$ (see results
shown after Lemma~\ref{Lem:betaroot})
this constant term is

\hspace{0.5in} $-\beta\sum_{j=0}^t b_j \bfs_{j-1} - 
\sum_{j=0}^t b_j \bfr_{j-1}$\hspace{0.2in} (use only the term for $m=0$)

\hspace{0.4in}
$ = -\sum_{j=0}^t b_j (\bfr_{j-1} + \beta \bfs_{j-1}) $

\hspace{0.45in}$ = -\sum_{j=0}^t b_j \bfs_{j}$ \hspace{1.35in} (since 
$\bfr = (L-\beta) \bfs$)

\hspace{0.4in}
$ = 0$ \hspace{1.0in}  \hspace{1.02in} (since $f$ annihilates $\bfs$).

For $m \ge 1$ we have:
$$
[x^m]H(x) = [x^{m-1}]P_{l+1,\bfs}(x) - \beta [x^m]P_{l+1,\bfs}(x) -[x^m]P_{l,\bfr}(x).
$$
From Equations~\ref{(1)} and~\ref{(2)}, we get
\begin{itemize}
\item[] $[x^{m-1}]P_{l+1,\bfs}(x) =
\sum_{j=0}^{t-m-l-1} b_{j+m+l+1} \bfs_{j+l+1} = 
 b_{m+l+1} \bfs_{l+1} + 
\sum_{j=1}^{t-m-l-1} b_{j+m+l+1} \bfs_{j+l+1}$

\item[] $[x^m]P_{l+1,\bfs}(x) \hspace{0.13in}  = 
\sum_{j=0}^{t-m-l-2} b_{j+m+l+2} \bfs_{j+l+1} = \hspace{0.3in}
 0  \hspace{0.37in} + 
 \sum_{j=1}^{t-m-l-1} b_{j+m+l+1} \bfs_{j+l}$

\item[] $[x^m]P_{l,\bfr}(x)  \hspace{0.3in} = 
\sum_{j=0}^{t-m-l-1} b_{j+m+l+1} \bfr_{j+l}  
\hspace{0.1in}  =  b_{t} \bfr_{t-m-1}  \hspace{0.18in} +
\sum_{j=0}^{t-m-l-2} b_{j+m+l+1} \bfr_{j+l}$.
\end{itemize}

The first and third of these equations come from removing the
first and last terms, respectively, from the sums, and the second 
equation comes from incrementing the index $j$ by $1$. 
Thus we have

\vspace{0.05in}	

\noindent $[x^m]H(x) =  b_{m+l+1} \bfs_{l+1} - b_t\ \bfr_{t-m-1} + 
\sum_{j=1}^{t-m-l-1} b_{j+m+l+1} (\bfs_{j+l+1} - 
 \beta \bfs_{j+l})
- \sum_{j=0}^{t-m-l-2}b_{j+m+l+1} \bfr_{j+l} $.

\vspace{0.1in}	

\noindent Since, $\bfr = (L-\beta)\bfs,\ \bfs_{j+l+1} - 
\beta \bfs_{j+l} = \bfr_{j+l}$ holds for all $j$ and so

\vspace{0.1in}

$[x^m]H(x) = b_{m+l+1} \bfs_{l+1} - b_t\ \bfr_{t-m-1}
+ \sum_{j=1}^{t-m-l-1}b_{j+m+l+1}\ \bfr_{j+l}
- \sum_{j=0}^{t-m-l-2}b_{j+m+l+1}\ \bfr_{j+l} $

\vspace{0.05in}

$\hspace{0.7in} = 
b_{m+l+1} \bfs_{l+1} - b_t\ \bfr_{t-m-1}
+ (b_{t}\ \bfr_{t-m-1}- b_{m+l+1}\ \bfr_{l})$

$\hspace{0.7in} = 
b_{m+l+1} (\bfs_{l+1} - \bfr_{l})
$

$\hspace{0.7in} = 
b_{m+l+1}\ (\beta \bfs_{l})$

$\hspace{0.7in} = 0$.

\noindent Thus $H(x) = 0$ and the claimed result holds. 
\end{proof}

\subsection{Getting longer and longer runs} \label{Sec:longer}
Let $\beta \in \FF_q$ be such that $(x-\beta,f(x)) = 1$.
Thus, there exists $g \in \FF_q[x]$ such that $g(x) (x-\beta) 
\equiv 1 \pmod{f}$ and so $g(L) = (L-\beta)^{-1}$ is a well-defined
function on $G(f)$.
Let $\bfr^{(0)}$ be a stream that has a run of 
zeroes of length $l$ starting at index $0$.
Inductively define a sequence of streams $\{\bfr^{(i)}\}_{i \ge 1}$ by $$\bfr^{(i)} = (L-\beta)^{-1}\bfr^{(i-1)},\ \ i \ge 1.$$
Equivalently, we have $\bfr^{(i-1)} = (L-\beta)\bfr^{(i)}$ for $i \ge 1$.

Let $P_{l,\bfr^{(0)}}$ be the polynomial defined for the run of $l$ 
zeroes in $\bfr^{(0)}$ and let $z$ be the multiplicity of $\beta$ 
as a root of $P_{l,\bfr^{(0)}}$. From Lemma~\ref{Lem:betaroot}, for 
$i \ge 0$, $\bfr^{(i+1)}$ has a run of zeros of length $l+i+1$ 
as long as $\bfr^{(i)}$ has a run of zeros of length $l+i$ 
starting at index $0$ and $\beta$ is a root of $P_{l+i,\bfr^{(i)}}$.
If $\beta$ is not a root of $P_{l+i,\bfr^{(i)}}$, then $\bfr^{(i+1)}$ 
does not have a run of $l+i+1$ zeroes starting at index $0$.
Repeated application of Lemma~\ref{Lem:betaroot} gives the 
following result.
Here, $\bfr^{(i)}, \bfr^{(i+1)}$ and $l+i$ take the roles of $\bfr, \bfs$ and $i$, respectively, from Section~\ref{Sec:StrOp}

\begin{lemma}\label{Lem:lotazeroes}
For $0 \le i \le z$, $\bfr^{(i)}$ has a run of zeros of length $l+i$ 
starting at index $0$ and $\bfr^{(i+1)}$ is {\it not} $0$ at index $0$.
\end{lemma}

\subsection{Starting at an index $\ne 0$}\label{Sec:non-zero}
Next we show that, for the results given in the previous sections, 
the requirement that the run of zeroes occurs at index $0$ is not 
essential. Specifically, suppose that a stream $\bfr$ has a run of $l$ 
zeroes starting at $n \in \mathbb{Z}$ and that $\bfr = (L-\beta)\bfs$. Then, 
$L^n\bfr = (L-\beta)L^n\bfs$; $L^n\bfr$ has a run of $l$ zeroes 
starting at index $0$ and so the results given in the previous 
section can be applied to $L^n\bfr$. Specifically, we would have:
\begin{itemize}
\item
$(\bfs_{n}, \bfs_{n+1},  \bfs_{n+2}, \ldots, \bfs_{n+l}) = 
(\bfs_{n}, \beta \bfs_{n}, \beta^2 \bfs_{n}, \ldots, \beta^l\bfs_{n})$;
\item
$\bfs_{n} = 0$ if and only if $\bfs$ has a run of $l+1$ zeroes 
starting at index $n$;
\item
if $P_{l,L^n\bfr}(x) = \sum_{j=l+1}^{t}b_j (\sum_{m=l}^{j-1}
\bfr_{n+m}x^{j-1-m})$; then, $\beta$ is a root of 
$\bfs_{n}f + P_{l,L^n\bfr}$;
\item
$\beta$ is root of $P_{l,L^n\bfr}$ if and only if $\bfs_n = 0$ 
if and only if $\bfs$ has a run of zeroes of length $l+1$ 
starting at index $n$;
\item
suppose $\bfs$ has a run of $l+1$ zeros starting at index $n$;
then $P_{l,L^n\bfr}(x) = (x-\beta)P_{l+1,L^n\bfs}(x)$.
\end{itemize}
These results will be used in the proof of the correctness of 
our new construction.

\section{The Generalized Runs (GR) construction of OOAs} \label{Sec:GR}
Let $f \in \FF_q[x]$ be a monic polynomial with non-zero constant 
term. For a stream ${\bf s} \in G(f)$ with period $\rho$ we define 
its {\em orbit} to be the set $\{{\bf s}, L{\bf s}, L^{2}{\bf s}, 
\ldots, L^{\rho-1}{\bf s} \}$ which is the set of the $\rho$ 
different sequences of length $\rho$ obtained by applying the 
left shift operator $L$ to ${\bf s} $ $i$ times, for 
$i=0, \ldots, \rho - 1$. The orbits partition $G(f)$ into $k$ 
orbits $C_1, C_2, \ldots, C_k$. All streams in a given orbit $C_i$ 
have the same period, denoted $\rho_i$. For each $C_i$ we choose 
one stream $\bfa^{(i)} \in C_i$, and call this the {\em base} for $C_i$.
Therefore, the elements of $C_i$ are all the shifts of the base 
$\bfa^{(i)}$.

For each $\beta \in \FF_q$ and $1 \le i \le k$, consider 
$\bfb = (L-\beta)\bfa^{(i)}$. Since $\bfb$ is annihilated by $f$, 
$\bfb \in G(f)$ and so lies in some orbit. Since any two 
elements of a given orbit are shifts of each other, the orbit 
that $\bfb$ lies in is independent of the choice of the base
$\bfa^{(i)}$ within the orbit $C_i$.

\subsection{RUNS construction from \cite{Castoldi}}\label{Sec:Runs}
In \cite{Castoldi} it was observed that if $f \in \FF_q[x]$ is a 
degree $t$ primitive polynomial and $\bfa$ is a fixed non-zero 
stream in $G(f)$ then, for each $\beta \in \FF_q^*$, there exists 
$k_{\beta} \in [1,q^t-1]$ such that, for all $n \in \ZZ$, the 
following holds:

\begin{equation} \label{(3)}
\bfa_{n+1} - \beta\bfa_{n} = \bfa_{n-k_{\beta}}.
\end{equation}

From this observation, an $OOA(q^t;t,q+1,t,q)$ over the alphabet 
$\FF_q$ was constructed as follows: for each $\beta \in \FF_q^*$, 
a block with $t$ columns was constructed where the first column 
is $L^{k_{\beta}} \bfa$ and each subsequent column is a left 
shift by $k_{\beta}$ places of the previous column.

When $f$ is primitive there are exactly two orbits, one of which 
is the trivial orbit containing only the all zeroes stream and 
the other is the orbit containing $\bfa$. Equation~(\ref{(3)}) is 
equivalent to stating that $(L-\beta)\bfa = L^{-k_{\beta}} \bfa$ 
and so $(L-\beta)$ maps the non-trivial orbit to itself (and the 
trivial orbit to itself).

Since $L^{k_{\beta}} \bfa = (L-\beta)^{-1}\bfa$, a left shift of 
$k_{\beta}$ positions is equivalent to applying $(L-\beta)^{-1}$ 
to the stream. Hence, the $j$th column in the block is the first 
$q^t-1$ elements of
$$ (L-\beta)^{-j}\bfa. $$
Two additional special blocks are constructed and these $q+1$ 
blocks in total, each with $t$ columns per block, yield the 
$OOA(q^t;t,q+1,t,q)$.

\subsection{Generalized RUNS construction}\label{Sec:GR1}
By using the analysis in Section~\ref{Sec:Runs} to describe the 
construction in \cite{Castoldi}, we generalize this construction 
to include non-primitive polynomials satisfying some simple 
properties. For these polynomials, an $OOA(q^t;t,\gamma+1,t,q)$ 
over the alphabet $\FF_q$ is constructed, where
$\gamma$ 
is the number of $\beta \in \FF_q$ such that $(x-\beta, f(x)) = 1$.
Equivalently,  $\gamma$ is the number 
of $\beta \in \FF_q$ that are not roots of $f$.
When $f$ is 
primitive, $\gamma = q$ and the resulting $OOA(q^t;t,q+1,t,q)$ 
is essentially the one produced in \cite{Castoldi}.  

For $f \in \FF_q[x]$, a monic polynomial of degree $t$ with non-zero constant term, 
let $\Gamma_f  = \{\beta_1, \beta_2, \ldots, \beta_{\gamma}\} \subseteq 
\FF_q$ be the set of those $\beta$ such that $(x-\beta,f(x)) = 1$.
For each such $\beta$, $(x-\beta)^{-1} \pmod{f}$ exists and so 
$(L-\beta)^{-1}$ is a well-defined function on $G(f)$. 
 
For each $\beta_k \in \Gamma_f$, and each orbit $C_i$, $1 \le i \le k$, 
we construct a $\rho_i \by t$ matrix $M(i,\beta_k)$ with $j$th column defined by the first 
$\rho_i$ elements of
$$ (L-\beta_k)^{-j} \bfa^{(i)}. $$

For each orbit $C_i$, we also have a special $\rho_i \by t$ matrix 
$M(i,\infty)$ with $j$th column 
defined by the first $\rho_i$ elements of
$$ L^{j-1}\bfa^{(i)}. $$

Define a matrix $M(f)$ consisting of the submatrices $M(i,\beta_j), 
i \in \{1,2,\ldots, s\}, \beta_j \in \{\infty\} \cup \Gamma_f$ 
arranged as below, where $\gamma = |\Gamma_f|$.
\[
\begin{tikzpicture}[mymatrixenv]
\matrix[mymatrix] (m)  {
M(1,\infty) & M(1,\beta_1) & M(1,\beta_2) & \ldots & M(1,\beta_{\gamma})  \\
M(2,\infty) & M(2,\beta_1) & M(2,\beta_2) & \ldots & M(2,\beta_{\gamma})  \\
\ldots      & \ldots       & \ldots       & \ldots & \ldots \\
M(s,\infty) & M(s,\beta_1) & M(s,\beta_2)  & & M(s,\beta_{\gamma}) \\
};
\mymatrixbraceright{1}{1}{$C_1$}
\mymatrixbraceright{2}{2}{$C_2$}\\
\mymatrixbraceright{4}{4}{$C_s$}
\mymatrixbracebottom{1}{1}{$Bl_M(0)$}
\mymatrixbracebottom{2}{2}{$Bl_M(1)$}
\mymatrixbracebottom{3}{3}{$Bl_M(2)$}
\mymatrixbracebottom{5}{5}{$Bl_M(\gamma)$}
\end{tikzpicture}
\]

In the next section, we prove that $M(f)$ is an $OOA(q^t;t,\gamma+1,t,q)$.
In Section~\ref{Sec:Runs} we noted that two additional blocks were 
included in the RUNS construction. The Generalized RUNS construction 
shows that one of these blocks corresponds to the set of matrices 
$\{M(i,\infty)\ |\ 1 \le i \le s\}$ and the second corresponds to the 
set of matrices $\{M(i,0)\ |\ 1 \le i \le s\}$ (we note that 
$0 \in \Gamma_f$ always holds).

\section{Correctness of the construction} \label{Sec:proofOOA}
The following result follows from Theorem $2$ of \cite{Raaphorst}.
\begin{lemma}\label{Lem:raaph}
The following are equivalent:
\begin{enumerate}
\item[(1)] A set of $t$ columns $\{i_{1},\ldots,i_{t}\}$ is covered in $M(f)$.
\item[(2)] There is no row $R = (r_0,r_1,\ldots, r_{t(\gamma+1)-1})$ other than the all-zero row of $M(f)$ such that $r_{i_{1}}=\cdots=r_{i_{t}}=0$.
\end{enumerate}
\end{lemma}
One thing to note about the rows of the $M(i,\infty)$ submatrices is that their 
$s$th row, with rows indexed $0 \le s \le \rho_i-1$, is the first 
$t$ entries of $L^{s}\bfa_i$. Equivalently these are the entries 
$s,s+1, \ldots, s+t-1$ of $\bfa_i$. This is used in Case 1 in the 
following proof showing that the construction is valid.

\begin{theorem}\label{Th:main}
For $f \in \FF_q[x]$, a monic polynomial of degree $t$ with non-zero constant term, the matrix $M(f)$, as defined in Section~\ref{Sec:GR}, is an $OOA(q^t;t,\gamma+1,t,q)$.
\end{theorem}

\begin{proof}
	
With needed changes, the proof follows the one given in \cite{Castoldi} for primitive $f$.

Let $L$ be a left-justified set of $t$ columns in $M(f)$ and let $B$ 
be the $q^t \by t$ array induced by the columns of $L$. Partition 
$L$ into sets $L_0, L_1, \ldots, L_{\gamma}$ where the columns of 
$L_i$ are in the $i$th block $Bl_M(i)$ and let $l_i = |L_i|, 0 
\le i \le \gamma$. Thus, $0 \le l_i \le t$ and $t = \sum_{i=0}^{t}l_i$.
This implies $t - l_0 = \sum_{i=1}^{t}l_i$, which is used a 
number of times in the proof.

Let $R = (r_0, r_1, \ldots, r_{t-1})$ be a row of $B$. This row 
corresponds to some orbit $C_i$ which, after a re-labeling of the 
orbits, we can assume to be $C_1$ having base $\bfa = \bfa^{(1)}$. This row 
corresponds to a row $S = (s_0, s_1, \ldots, s_{t(\gamma+ 1) - 1})$ 
of the full matrix $M(f)$.

Suppose that $S$ is not the all zeroes row. Hence, the 
 $\bfa$ 
is not the all zeroes stream. We claim that $R$ is not all zeroes.
By way of contradiction, suppose that it is all zeroes and consider
$3$ cases: Case 1: $l_0 > 0$, Case 2: $l_0 = 0$ and $s_0 \ne 0 $ 
and Case 3: $l_0 = 0$ and $s_0=0$.

In all cases, we find a run of zeroes of length $l\ge 0$ 
starting at some index $n$ and create the polynomial $P_{l}$ (we omit the $\bfa$ from the subscript of $P$ as $\bfa$ is fixed) of 
degree $t-l-1$ for this run. Then, letting $z_k$ be the multiplicity 
of $\beta_k$ as a root of $P_{l}$, we use that fact that 
$\sum_{k=1}^{\gamma}z_k \le \deg(P_l) = t - l - 1$ to derive the 
required contradiction.

Case 1: ($l_0 > 0$) :
Then, $(r_\omega,r_{\omega+1},\ldots,r_{\omega + l_0-1}) = (s_0,s_1,\ldots, s_{l_0-1})$ and 
from what we noted earlier about $M(1,\infty)$, this $l_0$-tuple is 
$(\bfa_{\omega},\bfa_{\omega+1}, \ldots, \bfa_{\omega+l_0-1})$, the first $l_0$ elements 
of $L^\omega\bfa$ where $\omega$ is the row index of $R$ in the matrix 
$M(1,\infty)$ (with indexing starting at $0$).

If $r_j \ne 0$ for some $0 \le j \le l_0-1$, we are done. So suppose 
that $r_0 = r_1 = \cdots = r_{l_0-1} = 0$. Thus, $\bfa$ has a run of 
zeroes of length $l' \ge l_0$ starting at index $n \le \omega$ and ending 
at index $n+l'-1$. Let $P_{l'}$ be our polynomial of degree $t-l'-1$ 
for this run of $l'$ zeroes starting at index $n$.

By Lemma~\ref{Lem:lotazeroes}, the $z_k$ entries in matrix 
$M(1,\beta_k)$ corresponding to the row $S$ are all zero while the 
$(z_k+1)$st entry is not. Thus, since $R$ is the all zeros row, 
we must have $l_k \le z_k$ for $1 \le k \le \gamma$, giving 
$\sum_{k=1}^{\gamma}l_k \le \sum_{k=1}^{\gamma}z_k$. Hence, 
$t - l_0 = \sum_{k=1}^{\gamma}l_k \le  \sum_{k=1}^{\gamma}z_k 
\le t - l' - 1$. This gives, $l_0 \ge l'+1$, a contradiction.

Case 2 ($l_0 = 0, s_0 \ne 0$) :
Then $\bfa$ contains a run of $l_0= 0$ zeroes starting at some index 
$n$. Let $P_{0}$ be our polynomial of degree $t-l_0-1 = t-1$ 
for this run of $0$ zeroes starting at index $n$.

By Lemma~\ref{Lem:lotazeroes}, the $z_k$ entries in matrix 
$M(1,\beta_k)$ corresponding to the row $S$ are all zero and the 
$(z_k+1)$st entry is not. Since $R$ is the all zeroes row, as 
before, we have $l_k \le z_k$ for all $k$.
 
Hence $\sum_{k=1}^{\gamma}l_k \le \sum_{k=1}^{\gamma}z_k$ and, 
since $t = \sum_{k=1}^{\gamma}l_k$, we have 
$t \le  \sum_{k=1}^{\gamma}z_k \le t - 1$, a contradiction.

Case 3 ($l_0 = 0, s_0 = 0$):
Since $S$ is not the all zeroes row, there exists a smallest 
$l \ge 1$ such that $s_0 = s_1 = \cdots = s_{l-1} = 0$, but 
$s_l \ne 0$. Thus, $\bfa$ has a run of $l' \ge l$ zeroes 
starting at some index $n \le 0$ and ending at index $n+l'- 1$.
Let $P_{l'}$ be our polynomial of degree $t-l'-1$ for this 
run of $l'$ zeroes starting at index $n$.

By Lemma~\ref{Lem:lotazeroes}, the $z_k$ entries in matrix 
$M(1,\beta_k)$ corresponding to the row $S$ are all zero and 
the $(z_k+1)$st entry is not. Since $R$ is the all zeroes row, 
we must have $l_k \le z_k$ for $1 \le k \le \gamma$. As in Case 2, 
we get $t = \sum_{k=1}^{\gamma}l_k \le  \sum_{k=1}^{\gamma}z_k 
\le t - l' - 1$ which gives $l' \le -1$, a  contradiction.

Thus all cases lead to a contradiction and so $R$ is not the 
all zeroes row and hence $L$ is covered by Lemma~\ref{Lem:raaph}.
Thus, $M(f)$ is an $OOA(q^t;t,\gamma+1,t,q)$ as claimed.
\end{proof}

As a corollary to this theorem, we have one of the main results
of~\cite{Castoldi}.

\begin{corollary}\label{Cor:prim}
An $OOA(q^t;t,q+1,t,q)$ over the alphabet $\FF_q$ can be derived 
from the LFSR defined by a degree $t \ge 2$ primitive polynomial 
$f \in \FF_q[x]$.
\end{corollary}
\begin{proof}
Since $f$ is primitive, $f$ has no roots in $\FF_q$ and so 
$(x-\beta,f(x)) = 1$ for all $\beta \in \FF_q$. Thus, 
$\Gamma_f = \FF_q$ and Theorem~\ref{Th:main} provides the 
required $OOA(q^t;t,q+1,t,q)$.
\end{proof}

\begin{example}
Let $f(x) = x^3 + x^2 + x + 1 = (x + 1)^3 \in \FF_2[x]$.
Then $G(f)$ is partitioned into $4$ orbits, as shown in 
the following table along with its (arbitrarily chosen) base.
\begin{small}
$$
\begin{array}{|c|l|c|} \hline
j & \hspace{0.1in}  a_j & period \\ \hline
0 & 0011\ldots          & 4 \\ \hline
1 & 01\ldots            & 2 \\ \hline
2 & 0\ldots             & 1 \\ \hline
3 & 1\ldots             & 1 \\ \hline
\end{array}
$$
\end{small}

Hence, $\Gamma_f = \{0\}$ and our OOA is the following $8 \by 6$ 
matrix, partitioned into the $8$ submatrices $M(i,\beta)$ for 
$(i,\beta) \in \{0,1,2,3\} \by \{\infty, 0\}$.

\vspace{0.1in}
\begin{tiny}
$$	
M(f)= 
\begin{array}{|c|c|} \hline  
M(0,\infty) & M(0,0) \\ \hline
M(1,\infty) & M(1,0) \\ \hline
M(2,\infty) & M(2,0) \\ \hline
M(3,\infty) & M(3,0) \\ \hline
\end{array}
=
\\
\begin{array}{|c|c|} \hline 
\begin{matrix}
0 & 0 & 1  \\
0 & 1 & 1 \\
1 & 1 & 0 \\
1 & 0 & 0 \\
\end{matrix}
&
\begin{matrix}
1 & 1 & 0 \\
0 & 1 & 1  \\
0 & 0 & 1 \\
1 & 0 & 0 \\
\end{matrix}\\ \hline

\begin{matrix}
0 & 1 & 0 \\
1 & 0 & 1 \\
\end{matrix}
&
\begin{matrix}
1 & 0 & 1 \\
0 & 1 & 0  \\
\end{matrix}\\ \hline

\begin{matrix}
0 & 0 & 0 \\
\end{matrix}
&
\begin{matrix}
0 & 0 & 0 \\
\end{matrix}\\ \hline
	
\begin{matrix}
1 & 1 & 1 \\
\end{matrix}
&
\begin{matrix}
1 & 1 & 1 \\
\end{matrix}\\ \hline	
\end{array}\ .
$$
\end{tiny}
\vspace{0.2in}
By Theorem~\ref{Th:main}, this is an $OOA(8;3,2,3,2)$
and is the one shown at the 
start of the paper.
\end{example}

\begin{example}
Let	$f(x) = x^4 + x^3 + 1 \in \FF_2[x]$.
Since $f$ is primitive, $G(f)$ is partitioned into $2$ orbits.
The following table shows their (arbitrarily chosen) bases.

\begin{small}	
$$
\begin{array}{|c|l|c|} \hline
j & \hspace{.6in}  a_j    & period \\ \hline
0 & 000111101011001\ldots & 15     \\ \hline
1 & 0\ldots               &  1      \\ \hline
\end{array}
$$
\end{small}		

\vspace{0.1in}
Hence, $\Gamma_f = \FF_2$ and our OOA is the following $16 \by 12$ 
matrix which is partitioned into the $6$ submatrices $M(i,\beta)$, 
one for each for $(i,\beta) \in \{0,1\} \by (\Gamma_f \cup \{\infty\})$.
	
\vspace{0.1in}
\begin{tiny}	
$$	
M(f)= 
\begin{array}{|c|c|c|} \hline  
M(0,\infty) & M(0,0) & M(0,1) \\ \hline
M(1,\infty) & M(1,0) & M(1,1) \\ \hline
\end{array}
= \begin{array}{|c|c|c|} \hline 

\begin{matrix}
0 & 0 & 0 & 1 \\
0 & 0 & 1 & 1 \\
0 & 1 & 1 & 1 \\
1 & 1 & 1 & 1 \\
1 & 1 & 1 & 0 \\
1 & 1 & 0 & 1 \\
1 & 0 & 1 & 0 \\
0 & 1 & 0 & 1 \\
1 & 0 & 1 & 1 \\
0 & 1 & 1 & 0 \\
1 & 1 & 0 & 0 \\
1 & 0 & 0 & 1 \\
0 & 0 & 1 & 0\\
0 & 1 & 0 & 0 \\
1 & 0 & 0 & 0 \\
\end{matrix}
&
\begin{matrix}
1 & 0 & 0 & 1 \\
0 & 1 & 0 & 0 \\
0 & 0 & 1 & 0 \\
0 & 0 & 0 & 1 \\
1 & 0 & 0 & 0 \\
1 & 1 & 0 & 0 \\
1 & 1 & 1 & 0 \\
1 & 1 & 1 & 1 \\
0 & 1 & 1 & 1 \\
1 & 0 & 1 & 1 \\
0 & 1 & 0 & 1 \\
1 & 0 & 1 & 0 \\
1 & 1 & 0 & 1 \\
0 & 1 & 1 & 0 \\
0 & 0 & 1 & 1 \\
\end{matrix}
	
&
	
\begin{matrix}
1 & 1 & 0 & 0 \\
1 & 0 & 1 & 0 \\
1 & 1 & 1 & 1 \\
1 & 0 & 0 & 0 \\
0 & 1 & 0 & 0 \\
1 & 1 & 1 & 0 \\
0 & 0 & 0 & 1 \\
1 & 0 & 0 & 1 \\
1 & 1 & 0 & 1 \\
0 & 0 & 1 & 1 \\
0 & 0 & 1 & 0 \\
1 & 0 & 1 & 1 \\
0 & 1 & 1 & 0 \\
0 & 1 & 0 & 1 \\
0 & 1 & 1 & 1 \\
\end{matrix}

\\ \hline
	
\begin{matrix}
0& 0 & 0 & 0 \\
\end{matrix}

&
	
\begin{matrix}
0 & 0 & 0 & 0 \\
\end{matrix}	
	
&
	
\begin{matrix}
0 & 0 & 0 & 0 \\
\end{matrix}

\\ \hline
\end{array}\ .
$$
\end{tiny}	

By Theorem~\ref{Th:main}, this is an $OOA(16;4,3,4,2)$ and is 
essentially the one that would be constructed by the RUNS 
construction of \cite{Castoldi}.
\end{example}

The OOAs constructed in Theorem~\ref{Th:main} correspond, by 
Theorem~\ref{tms-net}, to digital $(0,t,\gamma+1)$ nets in base $q$.

%\section{Comparison of constructions in $\FF_2[x]$}\label{Sec:ff2}
\section{Comparison of constructions}\label{Sec:ff2}
In this section we provide some data about the OOAs produced by 
our new construction. First, in Table~\ref{Comp2}, when the GR 
construction provides an OOA with $t(q+1)$ columns, as the RUNS 
construction and the Rosenbloom/Tsfasman/Skriganov (RTS) 
construction do, we compare the percentage of $t$-sets covered.  Following that we compute and compare the star-discrepancy of some $(t,m,s)$-nets from the GR and RTS constructions.

In this table, $\#$ is the number of OOAs produced by the method,
R$_{min}$, R$_{max}$, R$_{avg}$, R$_{SD}$, GR$_{min}$, GR$_{max}$, 
GR$_{avg}$ and GR$_{SD}$ give the minimum, maximum and average percentage
coverage as well as the standard deviation for the RUNS and GR
constructions, respectively.
For the RUNS construction, $\#$ is the number of primitive polynomial in $\FF_2[x]$ of degree $t$ and for the GR construction, $\#$ is the number of polynomials of degree $t$ with no roots in $\FF_2$.
The RTS construction produces exactly one OOA and the corresponding entry in the table is the percentage of $t$-sets covered for this OOA.
Boldfaced entries in the GR$_{max}$ column indicate 
when a non-primitive polynomial gives an OOA with coverage that 
is at least as good as when only primitive polynomials are used
and boldfaced entries in the GR$_{SD}$ column indicate then 
the standard deviation of the GR results is smaller than
the standard deviation of the R results. 

Secondly, for those OOAs produced by the GR construction which 
have the maximum number of $t$-sets of columns covered, 
Table~\ref{Poly2} provides the polynomial(s) that produce the 
corresponding OOA(s). In the ``Property'' column of these table, 
$P$ denotes a primitive polynomial, $I$ denotes an irreducible but 
not primitive polynomial and $R$ denotes a reducible polynomial, 
whose factorization is provided. Observe that when $t=8,9,10,11$, the polynomials that provide the best coverage are reducible.
This shows that, in general, primitivity is not essential for best coverage.

\iffalse
\begin{table}[H]
\caption{Comparison between GR, RUNS and RTS constructions in $\FF_2$.}
\label{Comp2}
\vspace{0.1in}	
\centering

\begin{small}
\begin{tabular}{|cc||cccc|cccc|cc|}\hline
$t$ & $cols$ & $\#$ & GR$_{min}$  & GR$_{max}$ & GR$_{avg}$ & $\#$ & 
R$_{min}$ & R$_{max}$ & R$_{avg}$ & $\#$ & RTS \\ \hline
2   &  6     &  1    & 0.7333    & 0.7333   & 0.7333   &    1  & 0.7333  & 0.7333  & 0.7333  & 1 & 0.7333     \\ \hline
3   &  9     &  2    & 0.5952    & 0.5952   & 0.5952   &    2  & 0.5952  & 0.5952  & 0.5952  & 1 & 0.4643  \\ \hline 
4   & 12     &  4    & 0.3556    & 0.5232   & 0.4323   &    2  & 0.4848  & 0.5232  & 0.5040  & 1 & 0.3455  \\ \hline
5   & 15     &  8    & 0.3413    & 0.4695   & 0.4186   &    6  & 0.3862  & 0.4695  & 0.4444  & 1 & 0.1968  \\ \hline 
6   & 18     &  16   & 0.2265    & 0.4465   & 0.3627   &    6  & 0.3891  & 0.4465  & 0.4100  & 1 & 0.1357  \\ \hline
7   & 21     &  32   & 0.3088    & 0.4237   & 0.3660   &   18  & 0.3088  & 0.4237  & 0.3631  & 1 & 0.0897  \\ \hline
8   & 24     &  64   & 0.0814    & \bf{0.4174}   & 0.3364   &   16  & 0.3471  & 0.3984  & 0.3777  &     1  & 0.0736  \\ \hline
\end{tabular}	
\end{small}

\end{table}
\fi

\begin{table}[H]
\caption{Comparison between GR, RUNS and RTS constructions in $\FF_2$.}
\begin{tiny}
\begin{center}
\begin{tabular}{|cc||ccccc|ccccc|cc|}\hline%\noalign{\smallskip}
$t$ & cols & $\#$ & GR$_{min}$  & GR$_{max}$ & GR$_{avg}$ & GR$_{SD}$ & $\#$ & R$_{min}$ & R$_{max}$ & R$_{avg}$ & R$_{SD}$ & $\#$ & RTS \\ \hline
 2 &   6 &   1 & 0.7333 & 0.7333       & 0.7333 & 0.0000      &   1 & 0.7333 & 0.7333 & 0.7333 & 0.0000 & 1 & 0.7333 \\ \hline
 3 &   9 &   2 & 0.5952 & 0.5952       & 0.5952 & 0.0000      &   2 & 0.5952 & 0.5952 & 0.5952 & 0.0000 & 1 & 0.4643 \\ \hline
 4 &  12 &   4 & 0.3556 & 0.5232       & 0.4323 & 0.0844      &   2 & 0.4848 & 0.5232 & 0.5040 & 0.0271 & 1 & 0.3455 \\ \hline
 5 &  15 &   8 & 0.3413 & 0.4695       & 0.4186 & 0.0551      &   6 & 0.3863 & 0.4695 & 0.4444 & 0.0326 & 1 & 0.1968 \\ \hline
 6 &  18 &  16 & 0.2265 & 0.4465       & 0.3627 & 0.0726      &   6 & 0.3891 & 0.4465 & 0.4100 & 0.0235 & 1 & 0.1357 \\ \hline
 7 &  21 &  32 & 0.3088 & 0.4237       & 0.3660 & \bf{0.0305} &  18 & 0.3088 & 0.4237 & 0.3631 & 0.0325 & 1 & 0.0897 \\ \hline
 8 &  24 &  64 & 0.0814 & {\bf 0.4174} & 0.3364 & 0.0816      &  16 & 0.3471 & 0.3984 & 0.3776 & 0.0141 & 1 & 0.0736 \\ \hline
 9 &  27 & 128 & 0.1109 & {\bf 0.3947} & 0.3280 & 0.0491      &  48 & 0.1955 & 0.3726 & 0.3273 & 0.0400 & 1 & 0.0391 \\ \hline
10 &  30 & 256 & 0.0777 & {\bf 0.3716} & 0.3129 & 0.0536      &  60 & 0.2525 & 0.3698 & 0.3222 & 0.0337 & 1 & 0.0255 \\ \hline
11 &  33 & 512 & 0.1490 & {\bf 0.3565} & 0.3109 & 0.0345      & 176 & 0.1903 & 0.3560 & 0.3154 & 0.0280 & 1 & 0.0160 \\ \hline
\noalign{\smallskip}
\end{tabular}
%\caption{$q = 2$} \label{table-q-21}
\label{Comp2}
\end{center}
\end{tiny}
\end{table}

\begin{table}[H]
\caption{Polynomials in $\FF_2[x]$ giving maximum coverage.}
\label{Poly2}
\vspace{0.1in}
\centering
\begin{small}
\begin{tabular}{|cc||l|c|l|}\hline
$t$ & cols & Polynomial                    & Property  & Factorization \\ \hline
 2   &  6      &  $x^{2} + x + 1$            & P         &    \\ \hline
 3   &  9      & $x^3+x  +1$                 & P         &    \\  
	&         & $x^3+x^2+1$                  & P         &    \\ \hline 
 4   & 12      & $x^4+x^3+1$                 & P         &    \\ \hline 
 5   & 15      & $x^5+x^4+x^3+x+1$           & P         &    \\ \hline 
 6   & 18      & $x^6 + x^4 + x^3 + x + 1$   & P         &    \\ \hline
 7   & 21      & $x^7 + x^5 + x^2 + x + 1$   & P         &    \\ \hline
 8   & 24      & $x^8 + x^6 + x^4 + x^3 + 1$ & R         & $(x^2 + x + 1) (x^6 + x^5 + x^4 + x + 1)$ \\ \hline
 9   & 27      & $x^9 + x^7 + x^6 + x^4 +x^2 + x + 1$   & R       & $(x^2 + x + 1)(x^7 + x^6 + x^5 + x^4 + 1)$
   \\ \hline
10   & 30      & $x^{10} + x^9 + x^8 + x^7 + x^4+ x^2 +  1$  & R & $(x^3 + x^2 + 1)(x^7 + x^5 + x^4 + x^3 + 1)$
  \\ \hline
11   & 33      & $x^{11} + x^9 + x^7 + x^6 + 1$   & R     & 
$(x^2 + x + 1)^3(x^5 + x^4 + x^2 + x + 1)$
 \\ \hline

\end{tabular}
\end{small}
\end{table}

For some applications, there may not be a need for the OOA to have
the maximum number of columns, but it may be that the percentage of 
covered $t$-sets of columns is more important. Table~\ref{Cov2} 
provide statistics on the coverage in the OOAs produced by the 
GR construction when the maximum number of columns is not produced.

\begin{table}[H]	
\caption{Percentage of covered $t$-sets for OOAs using GR construction in $\FF_2$.}
\label{Cov2}
\vspace{0.1in}	
\centering
	\begin{small}
\begin{tabular}{|cc||cccc|}\hline
$t$ & cols & $\#$ & GR$_{min}$  & GR$_{max}$ & GR$_{avg}$  \\ \hline
2   &  4     &  1    & 0.6667    & 0.6667   & 0.6667   \\ \hline
3   &  6     &  2    & 0.4000    & 0.6000   & 0.5000   \\ \hline 
4   &  8     &  4    & 0.2286    & 0.6286   & 0.4893   \\ \hline
5   & 10     &  8    & 0.1270    & 0.6190   & 0.4719   \\ \hline
6   & 12     &  16   & 0.2265    & 0.5541   & 0.4467   \\ \hline
7   & 14     &  32   & 0.0373    & 0.5035   & 0.3957   \\ \hline
8   & 16     &  64   & 0.0199    & 0.4867   & 0.3717   \\ \hline
\end{tabular}
\end{small}	
\end{table}

\noindent 
In Appendix~\ref{App:stat_analysis} we provide similar analyses for 
$\FF_q[x]$ when $q = 3,4,5,7,8,9$.

Table~\ref{star_discrepancy} shows the star-discrepancy of some $(t,m,s)$-nets from the RTS construction and the best (smallest) star-discrepancy of some of the $(t,m,s)$-nets from the GR construction.  The minimum of the GR construction is never worse than the RTS construction and is usually better, sometimes substantially.
\begin{table}[H]	
\caption{The star-discrepancy of the $(t,m,s)$-nets using the RTS and GR constructions.}
\label{star_discrepancy}
\vspace{0.1in}	
\centering  
\begin{tabular}{|cc||cc|}\hline
$q$ & $t$ & RTS & GR$_{max}$ \\ \hline 
2 & 2 & 0.5781 & 0.5781 \\ \hline
2 & 3 & 0.4531 & 0.3301 \\ \hline
2 & 4 & 0.2349 & 0.1855 \\ \hline
2 & 5 & 0.1621 & 0.1089 \\ \hline
2 & 6 & 0.0818 & 0.0648 \\ \hline \hline
3 & 2 & 0.4193 & 0.4193 \\ \hline
3 & 3 & 0.2310 & 0.1497 \\ \hline
3 & 4 & 0.1139 & 0.0757 \\ \hline
3 & 5 & 0.0475 & 0.0327 \\ \hline \hline
4 & 2 & 0.9799 & 0.7897 \\ \hline
4 & 3 & 0.9767 & 0.7111 \\ \hline \hline
5 & 2 & 0.3108 & 0.2854 \\ \hline
5 & 3 & 0.1485 & 0.0823 \\ \hline
  \end{tabular}
\end{table}

\section{Summary}
\label{Sec:summ}
In this paper we generalize the construction of OOAs given in
\cite{Castoldi} from using LFSRs defined by primitive polynomials to
LFSRs using a far larger class of polynomials. While doing so, we also
significantly simplify the proofs from \cite{Castoldi}.  For small
field size and small degree, we provide data on how close the OOAs
constructed are to being full OAs and compare the results to earlier
constructions. Additionally we compute the star-discrepancies of the
$(t,m,s)$-nets constructed which gives the accuracy for use in
quasi-Monte Carlo numerical integration.  In both measures the OOAs
and $(t,m,s)$-nets from the GR construction are at least the best
known and often better.  Provided examples show that primitivity is
not a requirement for the construction of OOAs having best
coverage. Indeed examples show that non-primitivity is sometimes a
requirement for the construction of OOAs having best coverage using
our method.

% \newpage

% \vspace{0.5cm}

\appendix
\section{Statistical data for OOA produced by GR in $\FF_q, q = 3,4,5,7,8,9$}
\label{App:stat_analysis}
As was done in Section~\ref{Sec:ff2}, we provide data about the OOAs 
produced by our new construction for larger fields. For some fields 
and polynomial degree, when the number of polynomials giving the OOA(s)
with best coverage is large, we do not provide the polynomials but 
only a count.

\begin{table}[H]
\caption{Comparison between GR and RUNS constructions in $\FF_3$.}
\begin{tiny}
\begin{center}
\begin{tabular}{|cc||ccccc|ccccc|cc|}\hline
$t$ & $cols$ & $\#$ & $GR_{min}$  & $GR_{max}$ & $GR_{avg}$ & $GR_{SD}$ & $\#$ & $R_{min}$ & $R_{max}$ & $R_{avg}$ & $R_{SD}$ & $\#$ & $RTS$ \\ \hline
2 &  8 &    3 & 0.7500 & 0.8214       & 0.7976 & 0.0412       &   2 & 0.8214 & 0.8214 & 0.8214 & 0.0000 & 1 & 0.7500 \\ \hline
3 & 12 &    8 & 0.7091 & {\bf 0.7409} & 0.7239 & {\bf 0.0159} &   4 & 0.7091 & 0.7409 & 0.7239 & 0.0172 & 1 & 0.5455 \\ \hline
4 & 16 &   24 & 0.5126 & 0.7027       & 0.6316 & {\bf 0.0438} &   8 & 0.5885 & 0.7027 & 0.6321 & 0.0458 & 1 & 0.3258 \\ \hline
5 & 20 &   72 & 0.5320 & {\bf 0.6654} & 0.6141 & 0.0369       &  22 & 0.6029 & 0.6607 & 0.6334 & 0.0182 & 1 & 0.2433 \\ \hline
6 & 24 &  216 & 0.2053 & {\bf 0.6346} & 0.5813 & 0.0548       &  48 & 0.4531 & 0.6338 & 0.5917 & 0.0403 & 1 & 0.2053 \\ \hline
7 & 28 &  648 & 0.4652 & {\bf 0.6169} & 0.5799 & 0.0262       & 156 & 0.4974 & 0.6152 & 0.5830 & 0.0212 & 1 & 0.1245 \\ \hline
8 & 32 & 1944 & 0.1493 & {\bf 0.6010} & 0.5671 & 0.0301       & 320 & 0.4627 & 0.5992 & 0.5683 & 0.0226 & 1 & 0.0972 \\ \hline
\noalign{\smallskip}
\end{tabular}
%\caption{$q = 3$} \label{table-q-3}
\end{center}
\end{tiny}
\end{table}

\begin{table}[H]
\caption{Polynomials in $\FF_3[x]$ giving maximum coverage.}
\label{Poly3}
\vspace{0.1in}
\centering
\begin{tiny}
\begin{tabular}{|cc||l|c|l|}\hline 
$t$ & $cols$ & Polynomial                    & Property & Factorization  \\ \hline
2   &  8     & $x^{2} + x + 2$                   & P    &                \\
	&        & $x^{2} + 2 x + 2$                 & P    &                \\	\hline
3   & 12     & $x^3+2x^2+1$                      & P    &                \\  
	&        & $x^3+x^2+2$                       & I    &                \\ \hline 
4   & 16     & $x^4+x^3+2x^2+2x+2$               & P    &                \\ 
	&        & $x^4+2x^3+2x^2+x+2$               & P    &                \\ \hline 
5   & 20     & $x^5 + 2x^3 + x^2 + 2x + 2$       & I    &                \\
	&        & $x^5 + 2x^3 + 2x^2 + 2x + 1$      & I    &                \\ \hline
6   & 24     & $x^6 + 2x^4 + x^3 + 2x^2 + x + 1$ & R    & $(x^2 + 1)(x^4 + x^2 + x + 1)$    \\
	&        & $x^6 + 2x^4 + 2x^3 + 2x^2 + 2x + 1$ & R  & $(x^2 + 1) (x^4 + x^2 + 2x + 1)$ \\ \hline
7   & 28     & $x^7 + x^6 + x^5 + x^3 + x + 1$   & R  & $(x + 2)(x^2 + 1)(x^4 + 2x^3 + 2x^2 + x + 2)$  \\
    &        & $x^7 + 2x^6 + x^5 + x^3 + 2$      & R  & $(x^2 + 2x + 2)(x^5 + 2x^3 + 2x^2 + 2x + 1)$  \\  \hline
8   & 32     & $x^8 + x^5 + x^4 + 2x^3 + 2x^2 + 2x + 2$ & I & \\
    &        & $x^8 + 2x^5 + x^4 + x^3 + 2x^2 + x + 2$  & I &  \\ \hline
\end{tabular}
\end{tiny}
\end{table}

\begin{table}[H]
\caption{Percentage of covered $t$-sets for OOAs using GR construction in $\FF_3$.}
\label{Cov3}
\vspace{0.1in}
\centering
\begin{tiny}
\begin{tabular}{|ccc||ccc|}\hline
$t$ & $cols$ & $\#$ & GR$_{min}$  & GR$_{max}$ & GR$_{avg}$ \\ \hline
2   &  6     &   2    & 0.7333   & 0.7333   & 0.7333   \\
    &  4     &   1    & 0.6667   & 0.6667   & 0.6667   \\ \hline
3   &  6     &   8    & 0.5357   & 0.7262   & 0.6250   \\ 
    &  9     &   2    & 0.9000   & 0.9000   & 0.9000   \\ \hline
4   & 12     &  24    & 0.4808   & 0.7495   & 0.6535   \\
    &  8     &   6    & 0.2286   & 0.8571   & 0.5405   \\ \hline 
5   & 15     &  72    & 0.3613   & 0.7083   & 0.6037   \\
	& 10     &  18    & 0.2540   & 0.7698   & 0.6332   \\ \hline
6	& 18     & 216    & 0.1872   & 0.6740   & 0.6077   \\
	& 12     &  54    & 0.0693   & 0.7370   & 0.5853   \\
\hline
\end{tabular}
\end{tiny}
\end{table}

\newpage

\begin{table}[H]
\caption{Comparison between GR and RUNS constructions in $\FF_4$.}
\begin{tiny}
\begin{center}
\begin{tabular}{|cc||ccccc|ccccc|cc|}\hline
$t$ & $cols$ & $\#$ & $GR_{min}$  & $GR_{max}$ & $GR_{avg}$ & $GR_{SD}$ & $\#$ & $R_{min}$ & $R_{max}$ & $R_{avg}$ & $R_{SD}$ & $\#$ & $RTS$ \\ \hline
2 &  10 &   6 & 0.8667 & {\bf 0.8667} & 0.8667 & 0.0000       &   4 & 0.8667 & 0.8667 & 0.8667 & 0.0000 & 1 & 0.7566 \\ \hline
3 &  15 &  20 & 0.6505 &  0.8220     & 0.7585 & 0.0576       &  12 & 0.7604 & 0.8220 & 0.7912 & 0.0321 & 1 & 0.4879 \\ \hline
4 &  20 &  81 & 0.5317 & {\bf 0.7670} & 0.7279 & 0.0624       &  32 & 0.7397 & 0.7641 & 0.7524 & 0.0091 & 1 & 0.4017 \\ \hline
5 &  25 & 324 & 0.6165 & {\bf 0.7423} & 0.7141 & {\bf 0.0258} & 120 & 0.6165 & 0.7423 & 0.7133 & 0.0307 & 1 & 0.2268 \\ \hline
\noalign{\smallskip}
\end{tabular}
%\caption{$q = 4$} \label{table-q-4}
\end{center}
\end{tiny}
\end{table}

\begin{table}[H]
\caption{Polynomials in $\FF_4[x]$ giving maximum coverage.}
\label{Poly4}
\vspace{0.1in}
\centering
\begin{tiny}
\begin{tabular}{|cc||l|c|}\hline
$t$ & $cols$ & Polynomial          & Property \\ \hline
2 &  10 &  $x^{2} + \omega x + \omega$ & P \\
  &     &  $x^{2} + \left(\omega + 1\right) x + \omega + 1$ & P \\
  &     &  $x^{2} + x + \omega$ & P \\
  &     &  $x^{2} + x + \omega + 1$ & P \\
  &     &  $x^{2} + \omega x + 1$ & I \\
  &     &  $x^{2} + \left(\omega + 1\right) x + 1$ & I \\ \hline

3 &  15 &  $x^{3} + \omega x^{2} + \left(\omega + 1\right) x + \omega$ & P \\
  &     &  $x^{3} + \omega x^{2} + \left(\omega + 1\right) x + \omega + 1$ & P \\
  &     &  $x^{3} + \left(\omega + 1\right) x^{2} + \omega x + \omega$ & P \\
  &     &  $x^{3} + \left(\omega + 1\right) x^{2} + \omega x + \omega + 1$ & P \\
  &     &  $x^{3} + x^{2} + x + \omega$ & P \\
  &     &  $x^{3} + x^{2} + x + \omega + 1$ & P \\ \hline

4 &  20 &  $x^{4} + \omega x^{3} + \omega x^{2} + \left(\omega + 1\right) x + \omega + 1$ & I \\
  &     &  $x^{4} + \omega x^{3} + x^{2} + \omega x + 1$ & I \\
  &     &  $x^{4} + \left(\omega + 1\right) x^{3} + \left(\omega + 1\right) x^{2} + \omega x + \omega$ & I \\
  &     &  $x^{4} + \left(\omega + 1\right) x^{3} + x^{2} + \left(\omega + 1\right) x + 1$ & I \\
  &     &  $x^{4} + x^{3} + \omega x^{2} + \omega x + \omega + 1$ & I \\
  &     &  $x^{4} + x^{3} + \left(\omega + 1\right) x^{2} + \left(\omega + 1\right) x + \omega$ & I \\ \hline

5 &  25 &  $x^{5} + \omega x^{4} + x^{3} + \omega x^{2} + \omega$ & P \\
  &     &  $x^{5} + \left(\omega + 1\right) x^{4} + x^{3} + \left(\omega + 1\right) x^{2} + \omega + 1$ & P \\
  &     &  $x^{5} + x^{4} + \omega x^{3} + \omega x^{2} + \omega + 1$ & P \\
  &     &  $x^{5} + x^{4} + \left(\omega + 1\right) x^{3} + \left(\omega + 1\right) x^{2} + \omega$ & P \\
  &     &  $x^{5} + \omega x^{4} + \omega x^{3} + \left(\omega + 1\right) x^{2} + 1$ & I \\
  &     &  $x^{5} + \left(\omega + 1\right) x^{4} + \left(\omega + 1\right) x^{3} + \omega x^{2} + 1$ & I \\ \hline
\end{tabular}
\end{tiny}
\end{table}

% \vspace{1.0in}

\begin{table}[H]
\caption{Percentage of covered $t$-sets for OOAs using GR construction in $\FF_4$.}
\label{Cov4}
\vspace{0.1in}
\centering
\begin{tiny}
\begin{tabular}{|cc||cccc|}\hline
$t$ & $cols$ & $\#$ & GR$_{min}$  & GR$_{max}$ & GR$_{avg}$ \\ \hline     
2 &   8 &   3 & 0.7500 & 0.7500 & 0.7500 \\
  &   6 &   3 & 0.7333 & 0.7333 & 0.7333 \\ \hline
3 &  12 &  21 & 0.7545 & 0.7591 & 0.7558 \\
  &   9 &   6 & 0.8214 & 0.8214 & 0.8214 \\
  &   6 &   1 & 0.4000 & 0.4000 & 0.4000 \\ \hline
4 &  16 &  81 & 0.3868 & 0.7973 & 0.7263 \\
  &  12 &  27 & 0.3556 & 0.8364 & 0.7542 \\
  &   8 &   3 & 0.4714 & 0.4714 & 0.4714 \\ \hline
5 &  20 & 324 & 0.4717 & 0.7608 & 0.7135 \\
  &  15 & 108 & 0.3413 & 0.7812 & 0.7010 \\
  &  10 &  12 & 0.2540 & 0.7063 & 0.5714 \\ \hline
\end{tabular}
\end{tiny}
\end{table}

\newpage

\begin{table}[H]
\caption{Comparison between GR and RUNS constructions in $\FF_5$.}
\begin{tiny}
\begin{center}
\begin{tabular}{|cc||ccccc|ccccc|cc|}\hline
$t$ & $cols$ & $\#$ & $GR_{min}$  & $GR_{max}$ & $GR_{avg}$ & $GR_{SD}$ & $\#$ & $R_{min}$ & $R_{max}$ & $R_{avg}$ & $R_{SD}$ & $\#$ & $RTS$ \\ \hline
2 &  12 &  10 & 0.8485 & {\bf 0.8788} & 0.8667 & 0.01565       &   4 & 0.8788 & 0.8788 & 0.8788 & 0.0000 & 1 & 0.7576 \\ \hline
3 &  18 &  40 & 0.8100 & {\bf 0.8395} & 0.8243 & {\bf 0.0102} &  20 & 0.8100 & 0.8395 & 0.8244 & 0.0104 & 1 & 0.5735 \\ \hline
4 &  24 & 205 & 0.6763 & {\bf 0.8095} & 0.7808 & 0.02020       &  48 & 0.7768 & 0.8015 & 0.7879 & 0.0077 & 1 & 0.4496 \\ \hline
\noalign{\smallskip}
\end{tabular}
%\caption{$q = 5$} \label{table-q-5}
\end{center}
\end{tiny}
\end{table}
\begin{table}[H]
\caption{Polynomials in $\FF_5[x]$ giving maximum coverage.}
\label{Poly5}
\vspace{0.1in}
\centering
\begin{tiny}
\begin{tabular}{|cc||l|c|}\hline
$t$ & $cols$ & Polynomial          & Property \\ \hline
2 &  12 &  $x^{2} + x + 2$   & P \\
  &     &  $x^{2} + 2 x + 3$ & P \\
  &     &  $x^{2} + 3 x + 3$ & P \\
  &     &  $x^{2} + 4 x + 2$ & P \\
  &     &  $x^{2} + 2$       & I \\
  &     &  $x^{2} + 3$       & I \\
\hline
3 &  18 &  $x^{3} + 3 x + 2$     & P \\
  &     &  $x^{3} + 3 x + 3$     & P \\
  &     &  $x^{3} + x^{2} + 2$   & P \\
  &     &  $x^{3} + 4 x^{2} + 3$ & P \\
  &     &  $x^{3} + 2 x + 1$     & I \\
  &     &  $x^{3} + 2 x + 4$     & I \\
  &     &  $x^{3} + 2 x^{2} + 1$ & I \\
  &     &  $x^{3} + 3 x^{2} + 4$ & I \\
\hline
4 &  24 &  $x^{4} + x^{3} + 2 x^{2} + x + 3$     & I \\
  &     &  $x^{4} + 2 x^{3} + 3 x^{2} + 3 x + 3$ & I \\
  &     &  $x^{4} + 3 x^{3} + 3 x^{2} + 2 x + 3$ & I \\
  &     &  $x^{4} + 4 x^{3} + 2 x^{2} + 4 x + 3$ & I \\
\hline
\end{tabular}
\end{tiny}
\end{table}

\begin{table}[H]
\caption{Percentage of covered $t$-sets for OOAs using GR construction in $\FF_5$.}
\label{Cov5}
\centering
\vspace{0.1in}
\begin{tiny}
\begin{tabular}{|ccc||ccc|}\hline
$t$ & $cols$ & $\#$ & GR$_{min}$ & GR$_{max}$ & GR$_{avg}$    \\ \hline 
2 &  10 &   4   & 0.8667 & 0.8667 & 0.8667 \\
  &   8 &   6   & 0.7500 & 0.8214 & 0.7976 \\ \hline
3 &  15 &  44   & 0.7736 & 0.8505 & 0.8096  \\ 
  &  12 &  12   & 0.7455 & 0.7636 & 0.7530  \\ 
  &  9  &   4   & 0.7619 & 0.7619 & 0.7619  \\ \hline
4 &  20 & 204   & 0.7129 & 0.8233 & 0.7845 \\
  &  16 &  78   & 0.7412 & 0.8495 & 0.8067 \\
  &  12 &  12   & 0.7939 & 0.8364 & 0.8222 \\
  &   8 &   1   & 0.2286 & 0.2286 & 0.2286 \\ \hline
\end{tabular}
\end{tiny}
\end{table}

\newpage

\begin{table}[H]
\caption{Comparison between GR and RUNS constructions in $\FF_7$.}
\begin{tiny}
\begin{center}
\begin{tabular}{|cc||ccccc|ccccc|cc|}\hline
$t$ & $cols$ & $\#$ & $GR_{min}$  & $GR_{max}$ & $GR_{avg}$ & $GR_{SD}$ & $\#$ & $R_{min}$ & $R_{max}$ & $R_{avg}$ & $R_{SD}$ & $\#$ & $RTS$ \\ \hline
2 &  16 &  21 & 0.8917 & \bf{0.9083} & 0.9012 & 0.0085 &   8 & 0.9083 & 0.9083 & 0.9083 & 0.0000 & 1 & 0.7583 \\ \hline
3 &  24 & 112 & 0.8375 & \bf{0.8898} & 0.8630 & 0.0125 &  36 & 0.8533 & 0.8898 & 0.8671 & 0.0114 & 1 & 0.5830 \\ \hline
4 &  32 & 819 & 0.8175 & \bf{0.8643} & 0.8470 & 0.0104 & 160 & 0.8312 & 0.8643 & 0.8513 & 0.0089 & 1 & 0.4653\\ \hline
\noalign{\smallskip}
\end{tabular}
%\caption{$q = 7$} \label{table-q-7}
\end{center}
\end{tiny}
\end{table}
% \vspace{1.0in}

\begin{table}[H]
\caption{Polynomials in $\FF_7[x]$ giving maximum coverage.}
\label{Poly7}
\vspace{0.1in}
\centering
\begin{tiny}
\begin{tabular}{|cc||l|c|}\hline
$t$ & $cols$ & Polynomial          & Property \\ \hline
2 &  16 &  $x^{2} + x + 3$ & P \\
  &     &  $x^{2} + 2 x + 3$ & P \\
  &     &  $x^{2} + 2 x + 5$ & P \\
  &     &  $x^{2} + 3 x + 5$ & P \\
  &     &  $x^{2} + 4 x + 5$ & P \\
  &     &  $x^{2} + 5 x + 3$ & P \\
  &     &  $x^{2} + 5 x + 5$ & P \\
  &     &  $x^{2} + 6 x + 3$ & P \\
  &     &  $x^{2} + x + 6$   & I \\
  &     &  $x^{2} + 3 x + 6$ & I \\
  &     &  $x^{2} + 4 x + 6$ & I \\
  &     &  $x^{2} + 6 x + 6$ & I \\ \hline
3 &  24 &  $x^{3} + x^{2} + 5 x + 2$   & P \\
  &     &  $x^{3} + 2 x^{2} + 6 x + 2$ & P \\
  &     &  $x^{3} + 4 x^{2} + 3 x + 2$ & P \\
  &     &  $x^{3} + 3 x^{2} + 3 x + 5$ & I \\
  &     &  $x^{3} + 5 x^{2} + 6 x + 5$ & I \\
  &     &  $x^{3} + 6 x^{2} + 5 x + 5$ & I \\ \hline
4 &  32 &  $x^{4} + x^{3} + 6 x^{2} + 2 x + 5$   & P \\
  &     &  $x^{4} + 2 x^{3} + 3 x^{2} + 2 x + 3$ & P \\
  &     &  $x^{4} + 5 x^{3} + 3 x^{2} + 5 x + 3$ & P \\
  &     &  $x^{4} + 6 x^{3} + 6 x^{2} + 5 x + 5$ & P \\
  &     &  $x^{4} + 3 x^{3} + 5 x^{2} + 5 x + 6$ & I \\
  &     &  $x^{4} + 4 x^{3} + 5 x^{2} + 2 x + 6$ & I \\ \hline
\end{tabular}
\end{tiny}
\end{table}

% \vspace{1.0in}

\begin{table}[H]
\caption{Percentage of covered $t$-sets for OOAs using GR construction in $\FF_7$.}
\label{Cov7}
\centering
\vspace{0.1in}
\begin{tiny}
\begin{tabular}{|ccc||ccc|}\hline
$t$ & $cols$ & $\#$ & GR$_{min}$ & GR$_{max}$ & GR$_{avg}$ \\ \hline 
2   &  14 &   6 & 0.9121 & 0.9121 & 0.9121            \\
    &  12 &  15 & 0.8485 & 0.8788 & 0.8667             \\ \hline
3   & 21     & 132   & 0.8346   & 0.8925   & 0.8629   \\
	& 18     &  30   & 0.8407   & 0.8824   & 0.8652   \\
	& 15     &  20   & 0.8550   & 0.8879   & 0.8681   \\ \hline
4   &  28 & 804 & 0.8232 & 0.8666 & 0.8505 \\
    &  24 & 360 & 0.7987 & 0.8667 & 0.8459 \\
    &  20 &  60 & 0.8285 & 0.8716 & 0.8498 \\
    &  16 &  15 & 0.7396 & 0.8071 & 0.7809 \\ \hline
\end{tabular}
\end{tiny}
\end{table}

\newpage

\begin{table}[H]
\begin{tiny}
\caption{Comparison between GR and RUNS constructions in $\FF_8$.}
\begin{center}
\begin{tabular}{|cc||ccccc|ccccc|cc|}\hline
$t$ & $cols$ & $\#f$ & $GR_{min}$  & $GR_{max}$ & $GR_{avg}$ & $GR_{SD}$ 
& $\#f$ & $R_{min}$ & $R_{max}$ & $R_{avg}$ & $R_{SD}$ & $\#f$ & $RTS$ 
	\\ \hline
2 &  18 &  28 & 0.9346 & {\bf 0.9346} & 0.9346 & 0.0000 & 18 & 0.9346 & 0.9346 & 0.9346 & 0.0000 & 1 & 0.7583 \\ \hline
3 &  27 & 168 & 0.8776 & {\bf 0.8916} & 0.8846 & 0.0053 & 144 & 0.8776 & 0.8916& 0.8844 & 0.0053 & 1 & 0.4909 \\ \hline
\noalign{\smallskip}
\end{tabular}
%\caption{$q = 8$} \label{table-q-8}
\end{center}
\end{tiny}
\end{table}

\begin{table}[H]
\caption{Polynomials in $\FF_8[x]$ giving maximum coverage.}
\label{Poly8}
\centering
\vspace{0.1in}
\begin{tiny}
\begin{tabular}{|cc||c|c|}\hline
$t$ & $cols$ & $\#$ polynomials  & Property \\ \hline
2  &  18 & 18  & P \\
   &     & 10  & I \\
   &     &  0  & R \\ \hline
3  &  27 & 18  & P \\
   &     & 10  & I \\
   &     &  0  & R \\ \hline
\end{tabular}
\end{tiny}
\end{table}

%\vspace{-0.4cm}

\begin{table}[H]
\caption{Percentage of covered $t$-sets for OOAs using GR construction in $\FF_8$.}
\label{Cov8}
\centering
\vspace{0.1in}
\begin{tiny}
\begin{tabular}{|ccc||ccc|}\hline
$t$ & $cols$ & $\#$ & GR$_{min}$ & GR$_{max}$ & GR$_{avg}$ \\ \hline 
2 &  16 &   7 & 0.7583 & 0.7583 & 0.7583 \\
  &  14 &  21 & 0.9121 & 0.9121 & 0.9121 \\ \hline
3 &  24 & 203 & 0.8518 & 0.8928 & 0.8809 \\
  &  21 &  42 & 0.8789 & 0.8805 & 0.8797 \\
  &  18 &  35 & 0.8505 & 0.8615 & 0.8571 \\ \hline
\end{tabular}
\end{tiny}
\end{table}

% \newpage
\begin{table}[H]
\caption{Comparison between GR and RUNS constructions in $\FF_9$.}
\begin{tiny}
\begin{center}
\begin{tabular}{|cc||ccccc|ccccc|cc|}\hline
$t$ & $cols$ & $\#$ & $GR_{min}$  & $GR_{max}$ & $GR_{avg}$ & $GR_{SD}$ & $\#$ & $R_{min}$ & $R_{max}$ & $R_{avg}$ & $R_{SD}$ & $\#$ & $RTS$ \\ \hline
2 &  20 &  36 & 0.9158 & {\bf 0.9263} & 0.9216 & 0.0053&  16 & 0.9263 & 0.9263 & 0.9263 & 0.0000 & 1 & 0.7579 \\ \hline
3 &  30 & 240 & 0.8702 & {\bf 0.9086} & 0.9005 & 0.0107&  96 & 0.8704 & 0.9086 & 0.9015 & 0.0101 & 1 & 0.5872 \\ \hline
\noalign{\smallskip}
\end{tabular}
%\caption{$q = 9$} %\label{table-q-9}
\end{center}
\end{tiny}
\end{table}

\begin{table}[H]
\caption{Polynomials in $\FF_9[x]$ giving maximum coverage.}
\label{Poly9}
\vspace{0.1in}
\centering
\begin{tiny}
\begin{tabular}{|cc||c|c|}\hline
$t$ & $cols$ & $\#$ polynomials  & Property \\ \hline
2  &  20 & 16 & P \\
   &     &  4 & I \\
   &     &  0 & R \\ \hline
3  &  30 &  8 & P \\
   &     &  8 & I \\
   &     &  0 & R \\ \hline
\end{tabular}
\end{tiny}
\end{table}

%\vspace{-0.4cm}

\begin{table}[H]
\caption{Percentage of covered $t$-sets for OOAs using GR construction in $\FF_9$.}
\label{Cov9}
\vspace{0.1in}
\centering
\begin{tiny}
\begin{tabular}{|ccc||ccc|}\hline
$t$ & $cols$ & $\#$ & GR$_{min}$ & GR$_{max}$ & GR$_{avg}$ \\ \hline 
2 &  18 &   8 & 0.9346 & 0.9346 & 0.9346 \\
  &  16 &  28 & 0.8917 & 0.9083 & 0.9012 \\ \hline
3 &  27 & 296 & 0.5874 & 0.9121 & 0.8908 \\
  &  24 &  56 & 0.8933 & 0.9032 & 0.8991 \\
  &  21 &  56 & 0.8421 & 0.9105 & 0.8922 \\ \hline
\end{tabular}
\end{tiny}
\end{table}

\end{document}